\documentclass[10pt]{amsart}
\usepackage{amsmath}
 \usepackage{mathrsfs}
  \usepackage{paralist}
  \usepackage{graphics} 
  \usepackage{epsfig} 
\usepackage{graphicx}  \usepackage{epstopdf}
 \usepackage[colorlinks=true]{hyperref}
\hypersetup{urlcolor=blue, citecolor=red}

\newtheorem{theorem}{Theorem}[section]

\newtheorem{lemma}[theorem]{Lemma}
\newtheorem{proposition}{Proposition}

\theoremstyle{definition}

\newtheorem{remark}{Remark}

\begin{document}
\thispagestyle{empty}
\title[Determining a boundary coefficient]{Determining a boundary coefficient in a dissipative wave equation: uniqueness and directional Lipschitz stability}

\author{Ka\" \i s Ammari}
\address{UR Analysis and Control of PDE, UR 13ES64, Department of Mathematics, Faculty of Sciences of Monastir, University of Monastir,
5019 Monastir, Tunisia}
\email{kais.ammari@fsm.rnu.tn}

\author{Mourad Choulli}
\address{Institut \'Elie Cartan de Lorraine, UMR CNRS 7502, Universit\'e de Lorraine-Metz, Ile du Saulcy, F-57045 Metz cedex 1, France}
\email{mourad.choulli@univ-lorraine.fr}
  
\begin{abstract}
We are concerned with the problem of determining the damping boundary coefficient appearing in a dissipative wave equation from a single boundary measurement. We prove that the uniqueness holds at the origin provided that the initial condition is appropriately chosen. We show that the choice of the initial condition leading to uniqueness is related to a fine version of unique continuation property for elliptic operators. We also establish a Lipschitz directional stability estimate at the origin, which is obtained by a linearization process.
\end{abstract}



\maketitle

\tableofcontents

\section{Introduction}

Let $\Omega$ be  a $C^\infty$-smooth bounded domain of $\mathbb{R}^d$ with boundary $\Gamma$. We assume that $\Gamma$ can be partitioned into two disjoint closed subsets with non empty interior that are denoted by $\Gamma _0$ and $\Gamma _1$.

\smallskip
We set, where $\tau >0$ is fixed,
\[
Q=\Omega \times (0,\tau),\quad \Sigma _0=\Gamma_0 \times (0,\tau),\quad \Sigma _1=\Gamma_1 \times (0,\tau),
\]
and we consider the following initial-boundary value problem (abbreviated to IBVP in the sequel) for the wave equation:
\begin{equation}\label{ibvp1}
\left\{
\begin{array}{lll}
\partial _t^2 u - \Delta u = 0 \;\; &\textrm{in}\;   Q, 
\\
u = 0 &\textrm{on}\;  \Sigma _0, 
\\
\partial _\nu u +b(x)\partial _tu= 0 &\textrm{on}\;  \Sigma _1, 
\\
u(\cdot ,0) = u^0,\; \partial_t u (\cdot ,0) = v^0.
\end{array}
\right.
\end{equation}

We are interested in the inverse problem of determining the boundary coefficient $b$ from the boundary measurement $\partial _\nu {u_b}_{|\Sigma _1}$, where $u_b$ is the solution, if it exists, of the IBVP \eqref{ibvp1}. At least formally, if each  term in the left hand side of the third equation of \eqref{ibvp1} belongs to $L^2(\Sigma  _1)$, then the inverse problem under consideration is equivalent to determining $b$ from $b\partial _t{u_b}_{|\Sigma _1}$. Therefore this problem is highly non linear. 

\smallskip
Before stating our uniqueness and stability results, we need to reformulate the IBVP \eqref{ibvp1} as an abstract Cauchy problem. To this purpose, we set 
\[
V=\{w\in H^1(\Omega );\; w_{|\Gamma _0}=0\}. 
\]
Here $w_{|\Gamma _0}$ is to be understood in the usual trace sense. When it is equipped with the $H^1$-norm, $V$ is a Hilbert space. We note in addition that the Poincar\'e inequality holds true for $V$ and therefore $w\mapsto \|\nabla w\|_{L^2(\Omega )^d}$ defines an equivalent norm on $V$. 

\smallskip
For $s>0$ and $1\leq r\leq \infty$ , we introduce the vector space
\[
B_{s,r}(\mathbb{R}^{d-1} ):=\{ w\in \mathscr{S}'(\mathbb{R}^{d-1});\; (1+|\xi |^2)^{s/2}\widehat{w}\in 
L^r(\mathbb{R}^{d-1} )\},
\]
where $\mathscr{S}'(\mathbb{R}^{n-1})$ is the space of temperate distributions on $\mathbb{R}^{d-1}$ and
$\widehat{w}$ is the Fourier transform of $w$. Equipped with its natural norm 
\[
\|w\|_{B_{s,r}(\mathbb{R}^{d-1} )}:=\|(1+|\xi |^2)^{s/2}\widehat{w}\|_{ L^r(\mathbb{R}^{d-1} )},
\]
$B_{s,r}(\mathbb{R}^{d-1} )$ is a Banach space (it is noted that $B_{s,2}(\mathbb{R}^{d-1} )$ is merely the usual Sobolev
space $H^s(\mathbb{R}^{n-1} )$). By using local charts and a partition of unity, we construct 
$B_{s,r}(\Gamma _1)$ from $B_{s,r}(\mathbb{R}^{d-1})$ similarly as $H^s(\Gamma _1)$ is built 
from $H^s(\mathbb{R}^{d-1})$.

\smallskip
The main interest in this spaces is that the multiplication by a function from $B_{s,1}(\Gamma _1)$ defines a bounded operator on $H^s(\Gamma _1)$ (we refer to \cite{Ch}[Theorem 2.1, page 605] for more details). We set
\[
B_{1/2,1}^+(\Gamma _1)=\{ b\in B_{1/2,1}(\Gamma _1);\; 0\leq b\}.
\]

Let $b\in B_{1/2,1}^+(\Gamma _1)$. We define on $\mathcal{H}=V\times L^2(\Omega )$ the unbounded operator $\mathcal{A}_b$ as follows
\[
\mathcal{A}_b\left( \begin{array} {c} u\\v\end{array}\right) = \left(\begin{array}{cc}0&I\\ \Delta & 0\end{array}\right)\left( \begin{array} {c} u\\v\end{array}\right)
\]
with domain
\[
D(\mathcal{A}_b)=\left\{ \left( \begin{array} {c} u\\v\end{array}\right)\in \left[H^2(\Omega )\cap V\right]\times V;\; \partial _\nu u+bv=0\; \rm{on}\; \Gamma _1\right\}.
\]

From \cite{TW}[Proposition 3.9.2, page 109], $\mathcal{A}_b$ is an m-dissipative operator and therefore it generates on $\mathcal{H}$ a $C_0$-semigroup of contractions $e^{t\mathcal{A}_b}$. In that case, we have, for any integer $k\geq 1$,
\begin{equation}\label{reg1}
e^{t\mathcal{A}_b}\left( \begin{array} {c} u^0\\v^0\end{array}\right)\in  \bigcap_{j=0}^kC^{k-j}([0,\tau ];D(\mathcal{A}_b^j)),\;\; \left( \begin{array} {c} u^0\\v^0\end{array}\right)\in D(\mathcal{A}_b^k),
\end{equation}
with the convention that $\mathcal{A}_b^0=I$ and $D(\mathcal{A}_b^0)=\mathcal{H}$.

\smallskip
Moreover, we have the estimate
\begin{equation}\label{sgest}
\left\|e^{t\mathcal{A}_b}\left( \begin{array} {c} u^0\\v^0\end{array}\right)\right\|_{\bigcap_{j=0}^kC^{k-j}([0,\tau ];D(\mathcal{A}_b^j))}\leq C \left\|\left( \begin{array} {c} u^0\\v^0\end{array}\right)\right\|_{D(\mathcal{A}_b^k)}.
\end{equation}
Here the constant $C$  doesn't depend on $u^0$ and $v^0$.

\smallskip
Let $A: D(A) \subset L^2(\Omega )\rightarrow L^2(\Omega )$ be the unbounded operator given by
\[
A=\Delta ,\;\; D(A)=\{w\in H^2(\Omega )\cap V;\; \partial _\nu w_{|\Gamma _1}=0\}.
\]
The following observation will be useful in the sequel: for any integer $k\geq 1$ and $b\in B_{1/2,1}^+(\Gamma _1)$, we have
\begin{equation}\label{in}
D(A^k)\times \{0\}\subset D(\mathcal{A}_b^k).
\end{equation}

\smallskip
We are now ready to state our main results.

\begin{theorem}\label{theoremu}
Let $b\in B_{1/2,1}^+(\Gamma _1)$.
\\
(a) $\mathscr{I}=\left\{ \left( \begin{array} {c} u^0\\0\end{array}\right);\; u^0\in D(A) \;\; \textrm{and}\;\; \partial _tu_0\neq 0\; {\rm a.e.\; on}\; \Sigma _1\right\}\neq \emptyset$.
\\
(b) Let $\left( \begin{array} {c} u^0\\0\end{array}\right)\in \mathscr{I}$. Then ${\partial _\nu u_b}_{|\Sigma _1}={\partial _\nu u_0}_{|\Sigma _1}$ implies $b=0$.
\end{theorem}

\begin{theorem}\label{theorems}
We fix $b\in B_{1/2,1}^+(\Gamma _1)$ non identically equal to zero. We assume
\begin{equation}\label{sc1}
\left( \begin{array} {c} u^0\\v^0\end{array}\right)\in \bigcap_{0\leq \rho \leq 1}D(\mathcal{A}_{\rho b}^7)
\end{equation}
and 
\begin{equation}\label{sc2}
b\partial _tu_0\not\equiv 0.
\end{equation}
Then there exists $0<\rho _0 \leq 1$ so that
\[
\widetilde{\kappa} \| \rho b-0\|_{B_{1/2,1}(\Gamma _1)} \leq \| \partial _\nu u_{\rho b}-\partial_\nu u_0\|_{L^2(\Sigma _1)},\;\; 0\leq \rho \leq \rho_0.
\]
Here $\widetilde{\kappa}$ is a constant independent on $\rho$.
\end{theorem}

\begin{remark}
From the proof of Theorem \ref{theoremu} (a), we deduce 
\[
\mathscr{I}^\infty =\left\{ \left( \begin{array} {c} u^0\\0\end{array}\right);\; u^0\in C^\infty(\overline{\Omega}) \;\; \textrm{and}\;\; \partial _tu_0\neq 0\; {\rm a.e.\; on}\; \Sigma _1\right\}\neq \emptyset.
\]
Therefore, we can replace in Theorem \ref{theorems} \eqref{sc1} and \eqref{sc2} by 
\[
\left( \begin{array} {c} u^0\\v^0\end{array}\right)=\left( \begin{array} {c} u^0\\0\end{array}\right)\in \mathscr{I}^\infty .
\]
\end{remark}

The authors have already obtained in \cite{AC} a log-type stability estimate for the inverse problem consisting in determining both the potential and the damping coefficient in a dissipative wave equation from boundary measurements. These measurements correspond to all possible choices of the initial condition. The proofs in \cite{AC} are essentially based on observability inequalities for exactly controllable  systems and spectral decompositions.

\smallskip
The problem of determining a potential in a wave equation from the so-called Dirichlet-to-Neumann was map studied by many authors. This problem was initiated by Rakesh  and Symes \cite{RS}. We refer the reader who want to learn more on this problem to \cite{BCY} and references therein. 

\smallskip
The rest of this text is devoted to the proof of our main results. We prove Theorem \ref{theoremu} in Section \ref{section2} and Theorem \ref{theorems} in Section \ref{section3}.


\section{Proof of Theorem \ref{theoremu}}\label{section2}

We first prove a preliminary result. Henceforth, $\mathcal{L}^k$ denotes the $k$-dimensional Lebesgue measure. 


\begin{lemma}\label{lemma1a}
Let $\lambda \in \mathbb{R}$ and $u\in H^2(\Omega )\cap C^1(\overline{\Omega })$ satisfying 
\[
\Delta u+\lambda u= 0\; \textrm{in}\;  \Omega \;\;\; \textrm{and}\;\;\;  \partial _\nu u=0\ \; \textrm{on}\; \Gamma _1. 
\]
Then 
\[
\mathcal{L}^{d-1}(\{ x\in \Gamma _1;\; u(x)= 0\})=0.
\]
\end{lemma}

\begin{proof}[Sketch of the proof]
Since $\Omega$ is $C^\infty$-smooth. $\Gamma _i$ can covered by a finite number of open subsets $U$, where $U$ is such there exists a $C^\infty$-diffeomorphism $\psi:U\rightarrow B$, $B=B(0,1)$, satisfying
\[
\psi (U\cap \Omega )\subset B_+,\quad \psi (U\cap \Gamma )\subset B_0,
\]
with
\[
B_+=\{ x=(x',x_d)\in B;\; x_d>0\},\quad B_0=\{ x=(x',x_d)\in B;\; x_d=0\}.
\]
We set $v(y)=u(\psi ^{-1}(y))$, $y\in B_+$. Then $Pv=0$ in $ B_+$ and $\partial _dv=0$ on $B_0$. Here $P$ is a second order operator with $C^\infty$ coefficients. We extend $v$ to the whole of $B$ by setting $w=v$ in $B_+$ and $w(x',-x_d)=v(x',x_d)$, $(x',x_d)\in B_+$. We have $w\in H^2(\Omega )\cap C^1(\overline{B})$ because $\partial _dv=0$ on $B_0$ and $Qw=0$ in $B$, where $Q$ is a second order operator whose coefficients are obtained by taking the even extension of the coefficients of $P$. Checking the details of this construction, we see that $Q$ has Lipschitz coefficients. 

\smallskip
Let $\epsilon >0$ be given. If we denote by $\mathcal{H}^k$ the $k$-dimensional Hausdorff measure, we get by applying \cite{Ro}[Theorem 2, page 342] that 
\[
\mathcal{H}^{d-2+\epsilon}(\{z\in B;\; w(z)=0,\; \nabla w(z)=0\})=0.
\]
Therefore,
\[
\mathcal{H}^{d-2+\epsilon}(\{y\in B_0;\; v(y)=0,\; \nabla v(y)=0\})=0,
\]
In particular
\[
\mathcal{L}^{d-1}(\{y\in B_0;\; v(y)=0,\; \nabla v(y)=0\})=0.
\]
On the other hand by \cite{GT}[Lemma 7.7, page 152], $\nabla _{y'}v(\cdot ,0)=0$ a.e. in any set where $v(\cdot ,0)$ is constant. Hence
\[
\mathcal{L}^{d-1}(\{y\in B_0;\; v(y)=0,\; \partial _d v(y)=0\})=0.
\]
Bearing in mind that that a Lipschitz map preserves Lebesgue sets of zero measure, we get
\[
\mathcal{L}^{d-1}(\{x\in \Gamma _1;\; u(x)=0,\; \partial _\nu u(x)=0\})=0.
\]
But, $\partial _\nu u=0\in \Gamma _1$. Hence
\[
\mathcal{L}^{d-1}(\{x\in \Gamma _1;\; u(x)=0\})=0.
\]
\end{proof}

The following regularity theorem will be useful in the sequel. Since this result is not explicitly recorded in the literature, for sake of completeness we sketch its proof.
\begin{theorem}\label{theoremer}
Let $m\geq 0$ be an integer. For any $f\in H^m(\Omega )$, $g\in H^{m+3/2}(\Gamma _0)$ and $h\in H^{m+1/2}(\Gamma _1)$, the boundary value problem 
\[
\left\{
\begin{array}{lll}
-\Delta u = f \;\; &\mbox{in}\;   Q, 
\\
u =g &\mbox{in}\;  \Gamma_0 , 
\\
\partial _\nu u=h &\mbox{in}\;  \Gamma_1 .
\end{array}
\right.
\]
has a unique $u\in H^{m+2}(\Omega )$ satisfying
\begin{equation}\label{est}
\|u\|_{H^{2+m}(\Omega )}\leq C\left( \|f\|_{H^m(\Omega )}+ \|g\|_{H^{m+3/2}(\Gamma _0)}+\|h\|_{H^{m+1/2}(\Gamma _1)}\right).
\end{equation}
Here the constant $C$ is independent on $f$, $g$ and $h$.
\end{theorem}

\begin{proof}
Since there exists $E\in H^{m+2}(\Omega )$ such that $E=g$ on $\Gamma _0$ and $\partial _\nu E=h$ in $\Gamma$ with 
\[
\|E\|_{H^{2+m}(\Omega )}\leq C\left( \|g\|_{H^{m+3/2}(\Gamma _0)}+ \|h\|_{H^{m+1/2}(\Gamma _1)} \right).
\]
(e.g. for instance \cite{LM}[Theorem 8.3, page 39]), we  see, replacing $u$ by $u-E$ and $f$ by $f-\Delta E$, that it is enough to prove the theorem with $(g,h)=(0,0)$. We consider then the BVP
\begin{equation}\label{bvp1}
\left\{
\begin{array}{lll}
-\Delta u = f \;\; &\mbox{in}\;   Q, 
\\
u =0 &\mbox{in}\;  \Gamma_0 , 
\\
\partial _\nu u=0 &\mbox{in}\;  \Gamma_1 .
\end{array}
\right.
\end{equation}
Let $f\in L^2(\Omega )$. Because $w\in V\rightarrow \|\nabla w\|_{L^2(\Omega )^d}$ defines an equivalent norm on $V$, by the Lax-Milgram lemma, there exists a unique variational solution $u\in H^1(\Omega )$ of the BVP \eqref{bvp1}. That is 
\[
\int_\Omega \nabla u\cdot \nabla vdx=\int_\Omega fvdx\;\; \mathrm{for\, all}\; v\in V.
\]

\smallskip
Let $(\Omega _0,\Omega _1)$ be an open covering of an open neighborhood of $\overline{\Omega}$ such that $\Gamma _i\subset \Omega _i$ and $\Gamma_i \cap \Omega _{1-i}=\emptyset$, $i=0,1$. Let $(\psi _0,\psi _1)$ be a partition of unity subordinate to the covering $(\Omega _0,\Omega _1)$ with $\psi _i\in C_0^\infty (\Omega _i)$ and $\psi _i=1$ in an neighborhood of $\Gamma _i$, $i=1,2$.

\smallskip
Let $u_i=u\psi _i$, $i=0,1$. Then a straightforward computation shows that $u_0$ and $u_1$ are the respective solutions of the variational problems
\begin{align*}
&\int_\Omega \nabla u_0\cdot \nabla vdx=\int_\Omega f_0vdx\;\; \mathrm{for\, all}\; v\in H_0^1(\Omega ),
\\
&\int_\Omega \nabla u_1\cdot \nabla vdx=\int_\Omega f_1vdx\;\; \mathrm{for\, all}\; v\in H^1(\Omega ).
\end{align*}
Here 
\[
f_i=-2\nabla \psi_i \cdot \nabla u+\psi _if,\;\; i=0,1.
\]
Since the regularity theorem \cite{LM}[Theorem 5.4, page 165] is valid for both the Dirichlet and the Neumann BVP's, we obtain that $u_i\in H^2(\Omega )$ and
\[
\|u_i\|_{H^2(\Omega )}\leq C\|f_i\|_{L^2(\Omega )}\leq C'\left(\|u\|_{H^1(\Omega )}+\|f\|_{L^2(\Omega )}\right)\leq C''\|f\|_{L^2(\Omega )},\;\; i=0,1.
\]
Therefore $u=u_0+u_1\in H^2(\Omega )$ and 
\[
\|u\|_{H^2(\Omega )}\leq C\|f\|_{L^2(\Omega )}.
\]
Next, if $f\in H^1(\Omega )$ then $f_i\in H^1(\Omega )$, $i=0,1$. We can then repeat the previous argument to conclude that $u\in H^3(\Omega )$ and estimate \eqref{est} holds with $m=1$. We complete the proof by using an induction argument in $m$.
\end{proof}

\begin{proof}[Proof of Theorem \ref{theoremu}]
(a) Let $(\lambda _k)$ be the sequence of eigenvalues, counted according to their multiplicity, of the unbounded operator $-A$. Let $(\varphi _k)$ be a sequence of eigenfunctions forming an orthonormal basis of $L^2(\Omega )$, each  $\varphi _k$ corresponds to $\lambda _k$.

\smallskip
We note that, according to Theorem \ref{theoremer}, 
\[
\varphi_k\in \bigcap_{m\in \mathbb{N}}H^m(\Omega )=C^\infty (\overline{\Omega}).
\]

\smallskip
We fix $k$ and we take $\left( \begin{array} {c} u^0\\v^0\end{array}\right)=\left( \begin{array} {c} \varphi_k\\0\end{array}\right)$. Then $\left( \begin{array} {c} u^0\\v^0\end{array}\right)\in D(\mathcal{A}_0)\cap D(\mathcal{A}_b)$ and $u_0= \cos (\sqrt{\lambda _k}t)\varphi _k$ is the solution of the IBVP \eqref{ibvp1} corresponding to this particular choice of $\left( \begin{array} {c} u^0\\v^0\end{array}\right)$. 

\smallskip
We have $\partial _t u_0=-\sqrt{\lambda _k}\sin (\sqrt{\lambda _k}t)\varphi _k$. Hence $\partial _t u_0\neq 0$ a.e. on $\Gamma _1$ as an immediate consequence of Lemma \ref{lemma1a}. In other words, $\left( \begin{array} {c} \varphi_k\\0\end{array}\right)\in \mathscr{I}$.

\smallskip
(b) Let $(u^0,v^0)\in \mathscr{I}$. If $\partial _\nu {u_b}_{|\Sigma _1}=\partial _\nu {u_0}_{|\Sigma _1}=0$, then $b\partial _t{u_b}_{|\Sigma _1}=0$. Also, by the uniqueness of the solution of the IBVP \eqref{ibvp1}, we conclude that $u_b=u_0$. Consequently, 
\begin{equation}\label{ui}
b\partial _t{u_0}_{|\Sigma _1}=0.
\end{equation}
By our assumption the set where $\partial _t{u_0}_{|\Sigma _1}$ vanishes is of zero measure. Therefore, \eqref{ui} implies that $b=0$ a.e. on $\Gamma _1$.
\end{proof}


\section{Proof of Theorem \ref{theorems}}\label{section3}

We begin by proving an extension lemma.

 \begin{lemma}\label{lemma1}
(Extension Lemma) Let $k,\ell$ two non negative integers. For any $g\in C^\ell ([0,\tau ]; H^{k+1/2}(\Gamma _1))$, there exists $G\in C^\ell ([0,\tau ]; H^{k+2}(\Omega))$ so that, for any $t\in [0,\tau ]$,
\[
\left\{
\begin{array}{lll}
\Delta G(t) = 0 \;\; &\textrm{in}\;   Q, 
\\
G(t) = 0 &\textrm{on}\;  \Gamma_0 , 
\\
\partial _\nu G(t) =g(t) &\textrm{on}\;  \Gamma_1 .
\end{array}
\right.
\]
and
\begin{equation}\label{aest1}
\|G^{(j)}(t)\|_{H^{k+2}(\Omega )}\leq C \|g^{(j)}(t)\|_{H^{k+1/2}(\Gamma _1)},\;\; 0\leq j\leq \ell.
\end{equation}
Here the constant $C$ is independent on $g$.
\end{lemma}

\begin{proof}
Let $h\in H^{k+1/2}(\Gamma _1)$.  By Theorem \ref{theoremer} there exists a unique solution $Eh\in H^{k+2}(\Omega )$  of the BVP
\[
\left\{
\begin{array}{lll}
\Delta w = 0 \;\; &\textrm{in}\;   Q, 
\\
w = 0 &\textrm{on}\;  \Gamma_0 , 
\\
\partial _\nu w =h &\textrm{on}\;  \Gamma_1 .
\end{array}
\right.
\]
Moreover, we have the estimate
\begin{equation}\label{es}
\|Eh\|_{H^{k+2}(\Omega )}\leq C\|h\|_{H^{k+1/2}(\Gamma _1)},
\end{equation}
for some constant $C$ independent on $h$.

\smallskip
If $g\in C^\ell ([0,\tau ]; H^{k+1/2}(\Gamma _1))$, then, using that $E: h\in H^{k+1/2}(\Gamma _1)\rightarrow Eh\in H^{k+2}(\Omega )$ is linear bounded operator (the fact that $E$ is bounded  is a consequence of estimate \eqref{es}), it is straightforward to check that $G(t)=Eg(t)$ satisfies the required properties.
\end{proof}

\smallskip
Next, we consider the following non homogenous IBVP
\begin{equation}\label{ibvp2}
\left\{
\begin{array}{lll}
\partial _t^2 u - \Delta u = 0 \;\; &\textrm{in}\;   Q, 
\\
u = 0 &\textrm{on}\;  \Sigma _0, 
\\
\partial _\nu u =g &\textrm{on}\;  \Sigma _1, 
\\
u(\cdot ,0) = u^0,\; \partial_t u (\cdot ,0) = v^0.
\end{array}
\right.
\end{equation}

\begin{proposition}\label{proposition1}
We assume that $g\in C^3 ([0,\tau ]; H^{1/2}(\Gamma _1))$, $u^0\in H^2(\Omega )\cap V$, $v^0\in V$ and the compatibility condition
\begin{equation}\label{cc}
\partial _\nu u^0-g(\cdot ,0)=0\;\; \rm{on}\; \Gamma _1
\end{equation}
holds. Then the IBVP \eqref{ibvp2} has unique solution $u$ such that 
\[
\left( \begin{array} {c} u\\u'\end{array}\right)\in \mathscr{X}=C^1([0,\tau ],H^1(\Omega )\times L^2(\Omega ))\cap C([0,\tau ];H^2(\Omega )\times H^1(\Omega ))
\]
and
\begin{equation}\label{aest3}
\left\| \left( \begin{array} {c} u\\u'\end{array}\right)\right\|_{\mathscr{X}}\leq C\left( \left\| \left( \begin{array} {c} u^0\\v^0\end{array}\right)\right\|_{H^2(\Omega )\times H^1(\Omega )}+\|g\|_{C^3 ([0,\tau ]; H^{1/2}(\Gamma _1))}\right).
\end{equation}
Moreover, under the additional assumptions
\begin{equation}\label{aa}
g\in C^6([0,\tau ];H^{1/2}(\Gamma _1 )),\quad \left( \begin{array} {c} u^0-G(0)\\v^0-G'(0)\end{array}\right)\in D(\mathcal{A}_0^4),
\end{equation}
$u'\in C^3([0,\tau ];H^1(\Omega ))$ and
\begin{equation}\label{aest4}
\|u'\|_{C^3([0,\tau ];H^1(\Omega ))}\leq C\left( \left\| \left( \begin{array} {c} u^1\\v^1\end{array}\right)\right\|_{D(\mathcal{A}_0^4)}+\|g\|_{C^6 ([0,\tau ]; H^{1/2}(\Gamma _1))}\right).
\end{equation}
\end{proposition}

\begin{proof}
We denote by $G\in C^3([0,\tau ];H^2(\Omega ))$ the function given by Lemma \ref{lemma1} and corresponding to $g$. We observe that if $u$ is the solution of the IBVP \eqref{ibvp2} then, $v=u-G$ is the solution of following one

\begin{equation}\label{ibvp3}
\left\{
\begin{array}{lll}
\partial _t^2 v - \Delta v = F \;\; &\textrm{in}\;  Q, 
\\
v = 0 &\textrm{on}\;  \Sigma _0, 
\\
\partial _\nu v =0 &\textrm{on}\;  \Sigma _1, 
\\
v(\cdot ,0) = u^1,\; \partial_t u (\cdot ,0) = v^1.
\end{array}
\right.
\end{equation}
Here
\[
F=G'',\quad u^1=u^0-G(0),\quad v^1=v^0-G'(0).
\]
By the regularity assumptions on $u^0$, $v^0$ and $g$ and compatibility condition \eqref{cc}, we get that
\[
F\in C^1([0,\tau ];L^2(\Omega )),\quad \left( \begin{array} {c} u^1\\v^1\end{array}\right)\in D(\mathcal{A}_0).
\]

Therefore, the IBVP \eqref{ibvp3} has a unique solution $v$ so that
\[
\left( \begin{array} {c} v\\v'\end{array}\right)\in C^1([0,\tau ],\mathcal{H})\cap C([0,\tau ];D(\mathcal{A}_0)).
\]
This solution is given by
\begin{equation}\label{df}
\left( \begin{array} {c} v(t)\\v'(t)\end{array}\right)=e^{t\mathcal{A}_0}\left( \begin{array} {c} u^1\\v^1\end{array}\right)+\int_0^t e^{(t-s)\mathcal{A}_0}\left( \begin{array} {c} 0\\F(s)\end{array}\right)ds.
\end{equation}
In light of estimate \eqref{aest1}, we have
\begin{align}
&\left\| \left( \begin{array} {c} v\\v'\end{array}\right)\right\|_{C^1([0,\tau ],\mathcal{H})\cap C([0,\tau ];D(\mathcal{A}_0))}\label{aest2}
\\
&\hskip 3cm \leq C\left( \left\| \left( \begin{array} {c} u^0\\v^0\end{array}\right)\right\|_{H^2(\Omega )\times H^1(\Omega )}+\|g\|_{C^3 ([0,\tau ]; H^{1/2}(\Gamma _1))}\right).\nonumber
\end{align}
Since $u=v+G$, we deduce that $\left( \begin{array} {c} u\\u'\end{array}\right)\in\mathscr{X}$ and \eqref{aest2} implies \eqref{aest3}.

\smallskip
Next, we assume that the additional assumptions: 
\[
g\in C^6([0,\tau ];H^{1/2}(\Gamma _1 ),\quad \left( \begin{array} {c} u^1\\v^1\end{array}\right)\in D(\mathcal{A}_0^4),
\]
hold. Then we deduce from \eqref{df} that $u'\in C^3([0,\tau ];H^1(\Omega ))$ and \eqref{aest4} is satisfied.
\end{proof}

\begin{proof}[Proof of Theorem \ref{theorems}]
We make the following assumption
\[
\left( \begin{array} {c} u^0\\v^0\end{array}\right)\in \bigcap_{0\leq \rho \leq 1}D(\mathcal{A}_{\rho b}^7).
\]
According to regularity result \eqref{reg1}, we have
\[
u'_{\rho b}{_{|\Gamma _1}}\in C^6([0,\tau ];H^{1/2}(\Gamma _1)),\;\; 0\leq \rho \leq 1.
\]
and
\begin{equation}\label{aest5}
\|u'_{\rho b}{_{|\Gamma _1}}\|_{C^6([0,\tau ];H^{1/2}(\Gamma _1))}\leq C\left\|\left( \begin{array} {c} u^0\\v^0\end{array}\right)\right\|_{D(\mathcal{A}_0^7)},\;\; 0\leq \rho \leq 1.
\end{equation}

We see that $v_\rho =u_{\rho b}-u_0$ solves the IBVP \eqref{ibvp2} with $g=-\rho bu'_{\rho b}$. By using \eqref{aest4}, we get
\begin{equation}\label{aest6}
\| v'_\rho \|_{C^3([0,\tau ];H^{1/2}(\Gamma _1))}=\|u'_{\rho b}-u'_0\|_{C^3([0,\tau ];H^{1/2}(\Gamma _1))}\leq C\rho ,\;\; 0\leq \rho \leq 1.
\end{equation}

Next, let $w_0$ be the solution of the IBVP \eqref{ibvp2} corresponding to $u^0=v^0=0$ and $g=-bu'_0$. Then $z=u_{\rho b}-u_0-\rho w_0$ is the solution of the IBVP\eqref{ibvp2} corresponding to $u^0=v^0=0$ and $g=-\rho b(u'_{\rho b}-u'_0)$. Hence
\[
\| \partial _\nu z\|_{C^3([0,\tau ];H^{1/2}(\Gamma _1))}=\rho \|b(u'_{\rho b}-u'_0)\|_{C^3([0,\tau ];H^{1/2}(\Gamma _1))}.
\]
This estimate, in combination with \eqref{aest6}, yields
\[
\| \partial _\nu z\|_{C^3([0,\tau ];H^{1/2}(\Gamma _1))}\leq C\rho ^2,\;\; 0\leq \rho \leq \rho_0.
\]
Therefore,
\[
\lim_{\rho \downarrow 0}\frac{\partial _\nu u_{\rho b}-\partial_\nu u_0}{\rho}=-b\partial _tu_0\;\; \rm{in}\; C^3([0,\tau ];H^{1/2}(\Gamma _1))
\]
and then
\[
\lim_{\rho \downarrow 0}\frac{\partial _\nu u_{\rho b}-\partial_\nu u_0}{\rho}=-b\partial _tu_0\;\; \rm{in}\; L^2(\Sigma _1).
\]

By using $2\kappa =\|b\partial _tu_0\|_{L^2(\Sigma _1)}\neq 0$, we get
\[
\kappa \rho \leq \| \partial _\nu u_{\rho b}-\partial_\nu u_0\|_{L^2(\Sigma _1)},\;\; 0\leq \rho \leq \rho_0,
\]
for some $0<\rho _0\leq 1$.

\smallskip
We can rewrite this estimate as follows
\[
\widetilde{\kappa} \| \rho b-0\|_{B_{1/2,1}(\Gamma _1)} \leq \| \partial _\nu u_{\rho b}-\partial_\nu u_0\|_{L^2(\Sigma _1)},\;\; 0\leq \rho \leq \rho_0.
\]
This completes the proof.
\end{proof}


\section*{Acknowledgments}
We would like to thank Luc Robbiano who indicated to us the principle of the proof of Lemma \ref{lemma1a}.



\begin{thebibliography}{99}


\bibitem{AC}
\newblock K. Ammari and  M. Choulli
\newblock Stable determination of two coefficients in a dissipative wave equation from boundary measurements, 
\newblock arXiv:1403.3018.

\bibitem{BCY}
\newblock M. Bellassoued, M. Choulli and M. Yamamoto, 
\newblock Stability estimate for an inverse wave equation and a multidimensional Borg-Levinson theorem,  
\newblock \emph{J. Differ. Equat.}  \textbf{247} (2) (2009), 465-494.

\bibitem{Ch} 
\newblock M. Choulli, 
\newblock Stability estimates for an inverse elliptic problem, 
\newblock \emph{J. Inverse Ill-Posed Probl.} \textbf{10} (6) (2002), 601-610.

 \bibitem {GT} 
\newblock D. Gilbarg and N. S. Trudinger,
\newblock Elliptic partial differential equations  of second order,  
\newblock  2$^{nd}$ edition, Springer-Verlag, Berlin, 1983.
 
\bibitem{LM} 
\newblock  J.-L. Lions and E. Magenes, 
\newblock Non-homogenuous boundary value problems and applications I, 
\newblock Springer-Verlag, Berlin, 1972.

\bibitem{RS}
\newblock Rakesh and W. Symes,
\newblock Uniqueness for an inverse problems for the wave equation,
\newblock \emph{Commun. PDE} \textbf{13} (1988), 87-96.

\bibitem{Ro}
\newblock L. Robbiano, 
\newblock Dimension des z\'eros d'une solution faible d'un op\'erateur elliptique, 
\newblock \emph{J. Math. Pures et Appl.} \textbf{67} (1988), 339-357.

\bibitem{TW} 
\newblock M. Tucsnak and  G. Weiss,
\newblock  Observation and control for operator semigroups.
\newblock Birkh\"auser Advanced Texts, Birkh\"auser Verlag, Basel, 2009.

\end{thebibliography}
\end{document}